%% file: coupling-to-no-coupling-constraints.tex
\documentclass[a4paper,american,reqno]{amsart}

\input{setup}
\bibliography{coupling-to-no-coupling-constraints}

\begin{document}

\title{On Coupling Constraints in Linear Bilevel Optimization}
\author[D. Henke, H. Lefebvre, M. Schmidt, J. Thürauf]%
{Dorothee Henke, Henri Lefebvre, Martin Schmidt, Johannes Thürauf}

\date{\today}

\begin{abstract}
  \input{abstract}
\end{abstract}

\keywords{\input{keywords}}
\subjclass[2020]{\input{msc2020}}

\maketitle

\input{introduction}
\input{removing-ccs}
\input{discussion}
\input{acknowledgement}
\input{data-statement}

\printbibliography

\end{document}

%% file: setup.tex
\usepackage{babel}
\usepackage[utf8]{inputenc}
\usepackage[T1]{fontenc}
\usepackage{url}
\usepackage{xspace}
\usepackage{todonotes}
\usepackage{accents}
\usepackage{csquotes}
\usepackage{microtype}
\usepackage[style = authoryear,
            maxbibnames = 100,
            maxcitenames = 2,
            giveninits = true,
            isbn = false,
            url = true,
            backend = bibtex]{biblatex}
\usepackage[colorlinks,
            citecolor=blue,
            urlcolor=blue,
            linkcolor=blue]{hyperref} 

\makeatletter
\patchcmd{\@settitle}{\uppercasenonmath\@title}{\scshape\large}{}{}
\patchcmd{\@setauthors}{\MakeUppercase}{\scshape\normalsize}{}{}
\makeatother

\vfuzz \hfuzz
\newtheorem{theorem}{Theorem}[section]

\newtheorem{lemma}[theorem]{Lemma}
\newtheorem{corollary}[theorem]{Corollary}

\theoremstyle{definition}

\theoremstyle{remark}
\newtheorem{remark}[theorem]{Remark}

\input{macros}


%% file: macros.tex
\newcommand{\st}{\text{s.t.}}

\newcommand{\field}{\mathbb}

\newcommand{\reals}{\field{R}}

\newcommand{\R}{\reals}

\newcommand{\abbr}[1][abbrev]{#1.\xspace}

\newcommand{\eg}{\abbr[e.g]}

\newcommand{\ie}{\abbr[i.e]}

\newcommand{\wrt}{\abbr[w.r.t]}

\newcommand{\set}[1]{\{#1\}}

\newcommand{\norm}[1]{\left\lVert#1\right\rVert}

\newcommand{\rev}[1]{#1}


%% file: abstract.tex
It is well-known that coupling constraints in linear bilevel
optimization can lead to disconnected feasible sets, which is not
possible without coupling constraints.
However, there is no difference between linear bilevel problems with
and without coupling constraints \wrt\ their complexity-theoretical
hardness.
In this note, we prove that, although there is a clear difference
between these two classes of problems in terms of their feasible sets,
the classes are equivalent on the level of optimal solutions.
To this end, given a general linear bilevel problem with coupling
constraints, we derive a respective problem without coupling
constraints and prove that it has the same optimal solutions (when
projected back to the original variable space).


%% file: keywords.tex
Bilevel optimization,
Coupling constraints%
%
%

%% file: msc2020.tex
90Cxx
%
%

%% file: introduction.tex
\section{Introduction}
\label{sec:introduction}

The research interest in bilevel optimization problems increased
significantly over the last years and decades;
see, \eg, \textcite{Dempe_Zemkoho:2020} for a recent overview.
However, and although serious advances have been made both w.r.t.\
theoretical aspects and algorithmic developments, there are still
open questions.
This is even the case for linear bilevel optimization problems that we
consider in this note and that are given by\footnote{For the ease of
  presentation, we omit stating dimensions of matrices and vectors.}
\begin{subequations}
  \label{eq:bilevel-w-ccs}
  \begin{align}
    \min_{x \in X, y} \quad & c^\top x + d^\top y
    \\
    \st \quad & Ax + By \geq a,
    \label{eq:bilevel-w-ccs:ccs} \\
    & y \in S(x),
  \end{align}
\end{subequations}
where $S(x)$ is the set of optimal solutions to the $x$-parameterized
lower-level problem
\begin{equation}
  \label{eq:bilevel-ll}
  \min_{y} \quad f^\top y \quad \st \quad Cx + Dy \geq b
\end{equation}
\rev{and all variables are assumed to be continuous.}
In particular, we consider the optimistic linear bilevel problem, \ie,
the leader is able to choose the $y$ that is the best \wrt\ the
upper-level objective function if there are multiple optimal solutions
to the follower's problem.
Moreover, Problem~\eqref{eq:bilevel-w-ccs} contains coupling
constraints in~\eqref{eq:bilevel-w-ccs:ccs}, \ie, upper-level
constraints that explicitly depend on the lower-level variables.
These coupling constraints are the main topic of this note.
Instead, a bilevel problem with~$B = 0$ does not have any coupling
constraints.
Throughout the remainder of the paper, we assume that the bilevel
problem~\eqref{eq:bilevel-w-ccs} is solvable\rev{, that $X \subset
  \R^n$ is a given polyhedron, and that all vectors and matrices have
  rational entries. Note that the linear bilevel problem always has an optimal
  solution that is at a vertex of the polyhedron obtained by
  intersecting~$X$, the
  constraints in \eqref{eq:bilevel-w-ccs:ccs}, and in
  \eqref{eq:bilevel-ll}; see, e.g., \textcite{Bard:1998}.}

For motivating our main research question, we brief\/ly discuss
coupling constraints in the following both \wrt\ their impact on the
geometry of the bilevel feasible set and their impact on complexity.

\subsection{Geometry of the Feasible Set}
\label{sec:geom-feas-set}

As an example, we consider the following linear bilevel problem taken
from \textcite{Kleinert:2021}:
\begin{subequations}
  \label{eq:lbp-example-thomas-diss}
  \begin{align}
    \min_{x,y} \quad & F(x,y) = x + 6y \\
    \st \quad & -x + 5y \leq 12.5, \label{eq:lbp-example-thomas-diss:cc}\\
                     & y \in S(x),
  \end{align}
\end{subequations}
with $S(x)$ being the set of optimal solutions to the lower-level problem
\begin{subequations}
  \label{eq:lbp-example-thomas-diss-ll}
  \begin{align}
    \min_y \quad & f(x,y) = -y \\
    \st \quad & 2x -y \geq 0, \\
                 & -x -y \geq - 6,\\
                 & -x + 6y \geq -3,\\
                 & x + 3y \geq 3.
  \end{align}
\end{subequations}
Both levels are linear optimization problems and all variables are
continuous.

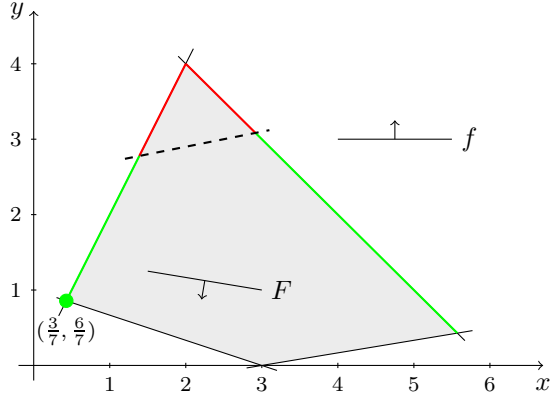
\begin{figure}
  \centering
  \input{tikz-imgs/lbp_graph.tikz}
  \caption{Visualization of
    Problem~\eqref{eq:lbp-example-thomas-diss}; mainly taken from
    \textcite{Kleinert:2021}.}
  \label{fig:lbp-example}
\end{figure}

The problem is visualized in Figure~\ref{fig:lbp-example}.
The gray area is the set of points that satisfy the lower-level
constraints.
The points above the dashed line are infeasible w.r.t.\ the
coupling constraint~\eqref{eq:lbp-example-thomas-diss:cc}.
Due to the optimization direction of the follower, the green faces
denote the bilevel feasible set.
Note that the two red faces are not part of the bilevel feasible set
since, for the respective $x$-values, the optimal replies~$y$ by the
follower violate the coupling constraint.

Hence, the example shows that the usage of coupling constraints makes
it possible to model bilevel feasible sets that are disconnected.
Theorem~3.3 in~\textcite{Benson:1989} states that the feasible set of
a linear bilevel problem with $B = 0$ is always connected, which means
that disconnected sets can only be modeled using coupling
constraints---an aspect that gained some prominence since it allows to
model mixed-binary linear problems using purely continuous linear
bilevel models; see, \eg, Section~3 in~\textcite{Vicente-et-al:1996}
and Section~3.1 in~\textcite{Audet-et-al:1997}.
Consequently, it seems to be the case that having coupling constraints
introduces larger modeling capabilities.

\subsection{Complexity}
\label{sec:complexity}

Already \textcite{Jeroslow:1985} showed that linear bilevel problems
are NP-hard.
An alternative proof using a reduction from \textsf{KNAPSACK} can be
found in \textcite{Ben-Ayed_Blair:1990}.
Moreover, \textcite{Hansen-et-al:1992} proved that linear bilevel problems
are strongly NP-hard by a reduction from \textsf{KERNEL}; see
Problem~\textsf{GT57} in \textcite{Garey_Johnson:1979}.
In particular, they even showed that the special case of min-max
problems without coupling constraints is strongly NP-hard.
The same holds true for the reduction from \textsf{3-SAT} shown in
\textcite{Marcotte_Savard:2005}.
In \textcite{Vicente-et-al:1994} it is further shown that even
checking local optimality of a given point is strongly NP-hard
(again via \textsf{3-SAT}).
As before, the authors do not require coupling constraints to achieve
this hardness result; see the linear bilevel problems in the proofs of
Theorems~5.1 and 5.2 in \textcite{Vicente-et-al:1994}.
Hence, coupling constraints are not required to make linear bilevel
problems strongly NP-hard.
Moreover, \textcite{Buchheim:2023} recently showed that the decision
versions of linear bilevel problems with coupling constraints are in
NP, which again implies that linear bilevel problems with coupling
constraints are not harder than those without.

\subsection{Research Question}
\label{sec:research-question}

Combining the two last discussions, we can summarize the following two
conclusions:
\begin{enumerate}
\item[(i)] Using coupling constraints in linear bilevel problems
  allows for modeling a richer class of feasible sets.
\item[(ii)] These stronger modeling capabilities do not result in any
  change of the hardness of the resulting problem in terms of
  complexity theory.
\end{enumerate}

Due to (ii), the question arises if we ``really'' increase
modeling capabilities by using coupling constraints.
While this is the case \wrt\ (dis-)connectedness of feasible sets, we
prove that there is no difference on the level of optimal solutions.
To do this, in the next section, given a linear bilevel problem with
coupling constraints, we derive a linear bilevel problem without
coupling constraints that has the same set of optimal solutions.


%% file: tikz-imgs/lbp_graph.tikz
\begin{tikzpicture}
  \def\xMin{0.0}
  \def\xMax{6.5}
  \def\yMin{0.0}
  \def\yMax{4.5}
  \def\axEps{0.2}
  \def\rrsColor{red}
  \def\bfsColor{green}

  \node (a) at (3/7, 6/7) {}; 
  \node (b) at (2, 4) {}; 
  \node (c) at (39/7, 3/7) {}; 
  \node (d) at (3, 0) {}; 
  \node (e) at (12.5/9, 25/9) {}; 
  \node (f) at (17.5/6, 18.5/6) {}; 

  \draw[->] (\xMin - \axEps, 0) -- (\xMax + \axEps, 0) node[below, yshift=-1.5pt] {$x$};
  \draw[->] (0, \yMin - \axEps) -- (0, \yMax + \axEps) node[left] {$y$};
  \foreach \x in {1,2,3,4,5,6}
  \draw (\x,1pt) -- (\x,-1pt)
  node[below] {\footnotesize{$\x$}};
  \foreach \y in {1,2,3,4}
  \draw (1pt,\y) -- (-1pt,\y)
  node[anchor=east] {\footnotesize{\y}};

  \fill[fill=gray!15] (a.center)--(b.center)--(c.center)--(d.center);
  \draw[domain=3/7 - 0.5*\axEps : 2 + 0.5*\axEps]
  plot (\x, 2.0 *\x);
  \draw[domain=2 - 0.5*\axEps : 39/7 + 0.5*\axEps]
  plot (\x, 6 - \x);
  \draw[domain=3 - \axEps : 39/7 + \axEps]
  plot (\x, - 0.5 + 1/6 * \x);
  \draw[domain=0.5 - \axEps : 3 + \axEps]
  plot (\x, 1 - 1/3 * \x);
  \draw[\rrsColor, thick] (e.center)--(b.center);
  \draw[\rrsColor, thick] (b.center)--(f.center);
  \draw[\bfsColor, thick] (a.center)--(e.center);
  \draw[\bfsColor, thick] (f.center)--(c.center);
  \draw[dashed, thick, domain=1.4 - \axEps : 2.9 + \axEps] plot (\x, 2.5 +
  1/5*\x);
  \draw[domain=4 : 5.5] plot (\x, 3) node[right] {$f$};
  \draw[->] (4.75, 3)--(4.75, 3.25);
  \draw[domain=1.5 : 3] plot (\x, 1.5 - 1/6 * \x) node[right]
  {$F$};
  \draw[->, rotate around={270:(2.25, 1.125)}] (2.25, 1.125)--(2.5, 1.08333);
  \draw[mark=*,mark size=2.5pt,mark options={color=green}] plot
  coordinates {(3/7,6/7)} node[below, yshift=-0.1cm]
  {\footnotesize{$(\frac{3}{7},\frac{6}{7})$}};


\end{tikzpicture}


%% file: removing-ccs.tex
\section{Exact Penalization of Coupling Constraints}
\label{sec:remov-coupl-constr}

In this section, we show that the bilevel
problem~\eqref{eq:bilevel-w-ccs} with coupling constraints can be
reformulated as a bilevel problem without coupling
constraints. To do this, we show that the violation of the coupling
constraints~\eqref{eq:bilevel-w-ccs:ccs} can be exactly penalized in
the objective function of the leader. This result is surprising for at
least two reasons. First, we just saw that the feasible region of
\eqref{eq:bilevel-w-ccs} may be disconnected and
nonconvex. In this case, standard Lagrangian duality theory is usually
limited and not as strong as in the convex case.
Second, the resulting problem, i.e., after penalization, has
a smooth objective function and no coupling constraint. Thus, it
differs from classic exact penalty methods, which often require a
nonsmooth penalty function.

Our key idea is to reformulate Problem~\eqref{eq:bilevel-w-ccs} in a
way so that the follower measures the violation of the coupling
constraints directly in the lower-level problem,
while the leader enforces that this violation is zero.
Doing so leads to the following bilevel problem which
contains a scalar and very simple coupling constraint:
\begin{subequations}
  \label{eq:bilevel-w-single-cc}
  \begin{align}
    \min_{x,y, \varepsilon} \quad
    & c^\top x + d^\top y
    \\
    \st \quad
    & x \in X, \\
    & \varepsilon = 0,
      \label{eq:bilevel-w-single-cc-single-cc}
    \\
    & (y, \varepsilon) \in \tilde{S}(x).
  \end{align}
\end{subequations}
Here, $\tilde{S}(x)$ is the set of optimal solutions to the $x$-parameterized
lower-level problem
\begin{subequations}
  \label{eq:bilevel-ll-single-ccs}
  \begin{align}
    \min_{y, \varepsilon} \quad
    & f^\top y \\
    \st \quad
    & Ax+By + \varepsilon e \geq a,
      \label{eq:bilevel-ll-single-ccs-fake-coupling}\\
    & Cx + Dy \geq b \label{eq:bilevel-ll-single-ccs-org-ll-cons}, \\
    & \varepsilon \geq 0,
      \label{eq:bilevel-ll-single-ccs-eps-nonnegative}
  \end{align}
\end{subequations}
where $e$ is the vector of ones.
Essentially, $\varepsilon$ is an additional variable of the
follower that measures the violation of the coupling constraints.
The newly introduced coupling
constraint~\eqref{eq:bilevel-w-single-cc-single-cc} enforces that it
equals zero. Most importantly, we have $S(x) = \text{proj}_y(\tilde
S(x))$ for all leader's decisions $x$.
In the next lemma, we show that Problem~\eqref{eq:bilevel-w-single-cc}
is, indeed, equivalent to Problem~\eqref{eq:bilevel-w-ccs}.
\begin{lemma}
  For every bilevel feasible point~$(x,y)$ of
  Problem~\eqref{eq:bilevel-w-ccs}, the point~$(x,y,0)$ is bilevel
  feasible for Problem~\eqref{eq:bilevel-w-single-cc} with the same
  objective value.
  For every bilevel feasible point~$(x,y,\varepsilon)$ of
  Problem~\eqref{eq:bilevel-w-single-cc}, the point~$(x,y)$ is
  bilevel feasible for Problem~\eqref{eq:bilevel-w-ccs} with the same
  objective value.
\end{lemma}
\begin{proof}
  Let~$(x,y)$ be a bilevel feasible point of
  Problem~\eqref{eq:bilevel-w-ccs}.
  Then, $(x,y,0)$ satisfies the upper- and lower-level constraints of
  Problem~\eqref{eq:bilevel-w-single-cc}.
  The point~$(y,0)$ is optimal for the lower-level
  problem~\eqref{eq:bilevel-ll-single-ccs} since $\varepsilon$ is not
  part of the lower-level objective function and
  $f^{\top}y$ is minimal regarding
  Constraints~\eqref{eq:bilevel-ll-single-ccs-org-ll-cons} due to the
  bilevel feasibility of $(x,y)$ for
  Problem~\eqref{eq:bilevel-w-ccs}.

  Let now $(x,y,\varepsilon)$ be a bilevel feasible point of
  Problem~\eqref{eq:bilevel-w-single-cc}.
  Then, the point~$(x,y)$ satisfies the upper- and lower-level
  constraints of
  Problem~\eqref{eq:bilevel-w-ccs} because $\varepsilon = 0$ holds.
  The point~$y$ is optimal for the lower-level
  problem~\eqref{eq:bilevel-ll} since every optimal
  solution~$\tilde{y}$ to the $x$-parameterized lower-level
  problem~\eqref{eq:bilevel-ll} can be extended to the feasible
  point~$(\tilde{y},\max \set{\norm{a-Ax-By}_{\infty}, 0})$ of the
  $x$-parameterized lower-level
  problem~\eqref{eq:bilevel-ll-single-ccs} with the same objective
  value.

  We finally note that Problem~\eqref{eq:bilevel-w-ccs} and
  Problem~\eqref{eq:bilevel-w-single-cc} have the same upper-level
  objective functions, which proves the claim.
\end{proof}

We now penalize the single coupling
constraint~\eqref{eq:bilevel-w-single-cc-single-cc} of
Problem~\eqref{eq:bilevel-w-single-cc} to obtain a bilevel
problem without coupling constraints.
Moreover, we show that there is a
polynomial-sized (in the bit-encoding length of the original problem's
data)
penalty parameter so that this formulation
is equivalent in terms of optimal solutions.

\begin{theorem}
  There is a polynomial-sized parameter~$\kappa > 0$ so that the
  bilevel problem (without coupling constraints)
  \begin{subequations} \label{eq:bilevel-wo-ccs}
    \begin{align}
      \min_{x,y, \varepsilon} \quad & c^\top x + d^\top y + \kappa
                                      \varepsilon
      \\
      \st \quad & x \in X, \ (y, \varepsilon) \in \tilde{S}(x),
    \end{align}
  \end{subequations}
  has the same set of optimal solutions as
  Problem~\eqref{eq:bilevel-w-single-cc}.
  Again, $\tilde{S}(x)$ is the set of optimal solutions to the
  $x$-parameterized lower-level
  problem~\eqref{eq:bilevel-ll-single-ccs}.
\end{theorem}
The idea of the proof is as follows.
First, we derive a single-level reformulation of the bilevel
problem~\eqref{eq:bilevel-w-single-cc}, using the KKT conditions of
the follower's problem \eqref{eq:bilevel-ll-single-ccs}.
Second, we apply results from augmented Lagrangian duality theory
for mixed-integer linear problems to show that a polynomial-sized
exact penalization parameter exists. Finally, we show that the resulting
mixed-integer linear program is nothing but the KKT reformulation of
Problem~\eqref{eq:bilevel-wo-ccs}.
\begin{proof}
  Since the lower-level problem~\eqref{eq:bilevel-ll-single-ccs}
  of Problem~\eqref{eq:bilevel-w-single-cc} is a
  linear program, we can replace it with its KKT
  conditions \parencite{Dempe-Dutta:2012}, leading to
  \begin{subequations}
    \label{kkt-formulation}
    \begin{align}
      \min_{x,y, \varepsilon} \quad
      & c^\top x + d^\top y
      \\
      \st \quad
      & x \in X, \ \varepsilon = 0,
      \\
      & Ax + By + \varepsilon e \geq a, \
        Cx + Dy \geq b, \
        \varepsilon \geq 0, \\
      & B^{\top} \lambda + D^{\top} \mu  = f, \ e^{\top} \lambda +
        \eta = 0, \\
      & \lambda, \mu, \eta \geq 0, \\
      &\lambda^{\top} (Ax + By + \varepsilon e - a) = 0, \
        \mu^{\top} (Cx + Dy - b) = 0, \ \eta \varepsilon = 0.
        \label{kkt-formulation:complementarity}
    \end{align}
  \end{subequations}
  Using additional binary variables~$z^{\lambda}, z^{\mu}, z^{\eta}$,
  and a sufficiently large big-$M$ value, we can reformulate the
  complementarity constraints~\eqref{kkt-formulation:complementarity}
  as the mixed-integer linear constraints
  \begin{subequations}
    \label{eq:linearized-compl}
    \begin{align}
      &\lambda \leq (e-z^{\lambda}) M, \ \mu \leq (e-z^{\mu}) M, \
        \eta \leq (1-z^{\eta}) M, \\
      &Ax + By +\varepsilon e -a \leq z^{\lambda} M, \
        Cx + Dy-b \leq z^{\mu} M, \ \varepsilon \leq z^{\eta} M.
    \end{align}
  \end{subequations}
  It is shown in \textcite{Buchheim:2023} that a valid and polynomial-sized
  value for $M$ can be computed in polynomial time.
  Thus, the resulting problem is a mixed-integer linear program whose
  input data is polynomial-sized in the bit-encoding length of the
  input data of Problem~\eqref{eq:bilevel-w-single-cc}.
  We now penalize the constraint~$\varepsilon = 0$ in the
  $\ell_{\infty}$-sense and obtain
  \begin{subequations}
    \label{penalty-kkt-formulation}
    \begin{align}
      \min_{x,y, \varepsilon} \quad
      & c^\top x + d^\top y + \kappa \varepsilon
      \\
      \st \quad
      & x \in X,
      \\
      & Ax + By + \varepsilon e \geq a, \
        Cx + Dy \geq b, \
        \varepsilon \geq 0, \\
      & B^{\top} \lambda + D^{\top} \mu  = f, \ e^{\top} \lambda +
        \eta = 0, \\
      & \lambda, \mu, \eta \geq 0, \ \eqref{eq:linearized-compl}.
    \end{align}
  \end{subequations}
  \rev{The existence of a finite and exact value for the penalty parameter~$\kappa$
  is guaranteed
  by Theorem~4 of~\textcite{Feizollahi2016},
  which states that the duality gap for the augmented Lagrangian dual
  of a solvable (mixed-integer) linear optimization problem
  can be closed by using a norm as the augmenting function
  and a sufficiently large but finite penalty parameter.}
  Moreover, Proposition~1  of~\textcite{Feizollahi2016}
  ensures that the sets of optimal solutions
  of~\eqref{penalty-kkt-formulation} and~\eqref{kkt-formulation} are the same.
  \rev{\textcite{Gu2020} show in Theorem~22 that
  the penalty parameter can be chosen to be of polynomial size
  in case of the $\ell_{\infty}$-norm.}
  Finally, Problem~\eqref{penalty-kkt-formulation} is the KKT
  reformulation of the bilevel problem~\eqref{eq:bilevel-wo-ccs}.
\end{proof}

\rev{Note that the results from the literature that we use to conclude
  the proof are not constructive but pure existence results.
  Hence, we also do not state an explicit formula or big-$O$ expression
  for the value or the size of the parameter~$\kappa$ here.}

\begin{corollary}
  \label{cor:final}
  There is a polynomial-sized penalty parameter~$\kappa > 0$ so
  that the following holds.
  For every bilevel optimal solution~$(x,y)$ to
  Problem~\eqref{eq:bilevel-w-ccs}, the point~$(x,y,0)$ is bilevel
  optimal to Problem~\eqref{eq:bilevel-wo-ccs} with the same
  objective value.
  For every bilevel optimal point~$(x,y,\varepsilon)$ of
  Problem~\eqref{eq:bilevel-wo-ccs}, the point~$(x,y)$ is
  bilevel optimal to Problem~\eqref{eq:bilevel-w-ccs} with the same
  objective value.
\end{corollary}

\begin{remark}
  \begin{enumerate}
  \item[(a)] Let us note that the reformulations
    of mixed-integer linear programs as linear bilevel problems
    presented in~\textcite{Vicente-et-al:1996}
    and~\textcite{Audet-et-al:1997}
    use two strategies for enforcing that
    continuous lower-level variables are binary.
    One is based on coupling constraints of the form~$v = 0$,
    where $v$ is an auxiliary lower-level variable, and the other one
    corresponds to an exact penalization of these coupling constraints
    and therefore requires no coupling constraints in the final model.
    Our approach to penalize general coupling constraints
    in linear bilevel problems uses the same idea in the final step.
  \item[(b)] The derivations in this section suggest that one could
    potentially extend the class of bilevel problems to which the
    results can be applied.
    First, the results in~\textcite{Gu2020} are valid for
    convex mixed-integer quadratic programs (MIQPs).
    Hence, this would allow for convex-quadratic objective functions
    of the leader and further integrality constraints in~$X$.
    Moreover, we can also allow for convex-quadratic (but still
    continuous) problems in the lower level since their KKT conditions
    lead to polyhedral constraints in the KKT
    reformulation~\eqref{kkt-formulation}.
    However, the big-$M$s from \textcite{Buchheim:2023} cannot be
    used directly anymore.
    Given that valid big-$M$s can also be found for quadratic programs
    in the lower level, the most general class of bilevel problems to
    which our results could be applied are those with convex MIQPs for
    the leader and convex QPs for the follower.
  \item[(c)] Due to the finiteness of the penalty parameter in
    Corollary~\ref{cor:final}, the original bilevel problem with
    coupling constraints can also be solved by a finite sequence of
    bilevel problems without coupling constraints if we follow
    standard ideas of penalty methods.
  \end{enumerate}
\end{remark}


%% file: discussion.tex
\section{Discussion}
\label{sec:discussion}

It has been known for at least 25~years that coupling constraints
in linear bilevel problems can lead to disconnected feasible sets and
that this is not possible without coupling constraints.
However, we prove that there is no difference between these two types
of linear bilevel problems on the level of optimal solutions.
While, on the one hand, this closes a gap in the literature on linear
bilevel optimization, it, on the other hand, also has some practical
consequences.
Many theoretical statements in the literature on bilevel optimization
are made and shown for problems without coupling constraints---either
simply for the ease of presentation or due to a lack of a proof for
the case with coupling constraints.
This note now allows for carrying over some of these results by
transforming the given problem having coupling constraints into one
without.
We also point to the open question of efficient computation of
the penalty parameter required in the last section.
While we prove that it is polynomial-sized (in the bit-encoding
length of the data of the given problem), the question on how to
compute it in polynomial time is still open.
Finally, the analogue question regarding the impact of coupling
constraints is, to the best of our knowledge, still open for
pessimistic bilevel problems.
\rev{We are rather convinced that the approach applied in this paper
  cannot be directly transferred to the pessimistic case.
  Hence, the latter is an interesting and important research
  question.}


%% file: acknowledgement.tex
\section*{Acknowledgements}

The second and third author acknowledge the support by the German
Bundesministerium für Bildung und Forschung within
the project \enquote{RODES} (Förderkennzeichen 05M22UTB).


%% file: data-statement.tex
\section*{Availability of data and materials}
\label{sec:avail-data-mater}

No data or code were generated or used during the study.
